\documentclass[12pt]{amsart}
\usepackage[mathscr]{eucal}
\usepackage{amscd}
\usepackage{amsfonts}
\usepackage{amsmath}
\usepackage{amsthm}
\usepackage{amssymb}
\usepackage{latexsym}
\usepackage[left=2cm,top=2cm,right=3cm,foot=2cm,headsep=1cm]{geometry}

\begin{document}

\newcommand{\nc}{\newcommand}
\newcommand{\bom}{{_{\mathbf{\omega}}}}
\newcommand{\st}{\divideontimes}
\def\neweq{\setcounter{theorem}{0}}
\newtheorem{theorem}[]{Theorem}
\newtheorem{proposition}[]{Proposition}
\newtheorem{corollary}[]{Corollary}
\newtheorem{lemma}[]{Lemma}
\newtheorem{example}[]{Example}
\theoremstyle{definition}
\newtheorem{definition}[]{Definition}
\newtheorem{remark}[]{Remark}
\newtheorem{conjecture}[equation]{Conjecture}
\newcommand{\dis}{{\displaystyle}}
\def\question{\noindent\textbf{Question.} }
\def\remark{\noindent\textbf{Remark.} }
\def\proof{\medskip\noindent {\textsl{Proof.} \ }}
\def\endproof{\hfill$\square$\medskip}
\def\str{\rule[-.2cm]{0cm}{0.7cm}}
\newcommand{\beq}{\begin{equation}\label}
\newcommand{\aand}{\quad{\text{\textsl{and}}\quad}}
\newcommand{\la}{\label}
\numberwithin{equation}{section}
\title{Double Hall algebras and derived equivalences }
\author{Tim Cramer}
\thanks{\emph{E-mail address}: tim.cramer@yale.edu}
\maketitle
We show that the reduced Drinfeld double of the Ringel-Hall algebra
of a hereditary category is invariant under derived equivalences. By
associating an explicit isomorphism to a given derived equivalence,
we also extend the results of Burban-Schiffmann (\cite{BS1,BS2}),
Sevenhant-Van den Bergh (\cite{SVdB}), and Xiao-Yang (\cite{XY}).

\section{Introduction}

Let $\mathcal{A}$ be an essentially small Abelian category such that
the sets Hom$(A,B)$ and Ext$^1(A,B)$ are each finite for all $A,B\in
\mathcal{A},$ and let $\mathcal{I}$ denote the set of isomorphism
classes of objects in $\mathcal{A}.$ Then it is possible to give the
vector space $\mathbb{C}[\mathcal{I}]$ the structure of an
associative algebra by the rule
\[[A]*[B]=\sum_{[C]\in \mathcal{I}} g_{A,B}^C [C],\]
where $g_{A,B}^C$ is defined to be the number of subobjects
$M\subset C$ such that $M \cong B$ and $C/M \cong A$. Equivalently,
\[g_{A,B}^C = \frac{|\Delta_{B, A}^C|}{|\mbox{Aut} A| |\mbox{Aut} B|},\]
where $\Delta_{B, A}^C$ denotes the set of short exact sequences
\[0 \rightarrow B \rightarrow C \rightarrow A \rightarrow 0.\]
The resulting algebra is known as the Hall algebra $H_{\mathcal{A}}$ of $\mathcal{A}.$ Our conditions on $\mathcal{A}$ guarantee that the structure constants $g_{A,B}^C$ are finite. Associativity can easily be shown, and indeed the structure constants $g_{A_1,A_2, \cdots, A_n}^B$ in the product
\[[A_1]*[A_2]*\cdots *[A_n]=\sum_{[B]\in \mathcal{I}} g_{A_1,A_2,\cdots A_n}^B [B]\]
count the number of filtrations
\[L_n\subset \cdots \subset L_2\subset L_1=B\]
such that $L_1/L_2\cong A_1,$ $L_2/L_3\cong A_2, \cdots$ and $L_n\cong A_n.$

Hall algebras first appeared in the work of Steinitz (\cite{S}) and
Hall (\cite{H}) in the case where $\mathcal{A}$ is the category of
Abelian $p$-groups. They reemerged in the work of Ringel
(\cite{R1}-\cite{R3}), who showed in \cite{R1} that when $\mathcal{A}$
is the category of quiver representations of an A-D-E quiver
$\vec{Q}$ over a finite field $\mathbb{F}_q$, the Hall algebra of
$\mathcal{A}$ provides a realization of the nilpotent subalgebra
$U_q(\mathfrak{n}_+)$ of the quantum group $U_q(\mathfrak{g})$
associated to the underlying graph of $\vec{Q}.$ More generally, if
$\vec{Q}$ is of affine type, then the subalgebra of the Hall algebra
generated by the simple objects corresponding to the vertices (known
as the ``composition subalgebra") is isomorphic to the nilpotent
subalgebra $U_q(\mathfrak{n}_+)$ of the quantum Kac-Moody algebra
$U_q(\mathfrak{g})$ associated to $\vec{Q}$.

In \cite{R1}, Ringel posed the question of how to extend this
construction naturally to recover the whole quantum group
$U_q(\mathfrak{g}).$ Using the group algebra of the Grothendieck group 
$K_0(\mathcal{A})$ to realize the torus algebra, he showed how to
extend the Hall algebra in such a way that it recovers the Borel
subalgebra $U_q(\mathfrak{b}_+)$ when
$\mathcal{A}=Rep_{\mathbb{F}_q}(\vec{Q})$. By generalizing the
coproduct of Green (\cite{Gr}) to this ``extended" Hall algebra, Xiao
(\cite{X2}) showed that it is a self-dual Hopf algebra when
$\mathcal{A}$ is the category of representations of a quiver. Using
the Drinfeld double construction, Xiao defined an algebra
$DH_{\mathcal{A}}$ that realizes the whole quantum group in the special
case that $\vec{Q}$ is an A-D-E quiver. In fact, this construction generalizes to any Abelian category $\mathcal{A}$ that is hereditary (i.e. of homological dimension less than or equal to one) and satisfies certain finiteness conditions.

It has remained an open question, however, whether this is the most natural way to
realize $U_q(\mathfrak{g}).$ Since the Hall algebras of such derived
equivalent categories as $Rep_{\mathbb{F}_q}(A_1^{(1)})$ and
$Coh(\mathbb{P}^1)$ have been found to correspond to positive
``halves" of the same quantum group (\cite{K1}), it has been thought
that the correct extension of $H_{\mathcal{A}}$ should be given in
terms of the derived category, or should at least be invariant under
derived equivalences. Indeed, many attempts have been made to define
associative algebras analogous to the Hall algebra for
$D^b(\mathcal{A})$ using exact triangles (e.g. \cite{X1}, \cite{T},
\cite{XXu}), but none of these constructions have realized the
quantum group $U_q(\mathfrak{g})$ for
$\mathcal{A}=Rep_{\mathbb{F}_q}(\vec{Q}).$ On the other hand, in the
case of the derived Bernstein-Gelfand-Ponomarev reflection functors
(\cite{GM})
\[R_{\alpha}:D^b(Rep_{\mathbb{F}_q}(\vec{Q}))\rightarrow D^b(Rep_{\mathbb{F}_q}(\sigma_{\alpha}\vec{Q})),\]
explicit isomorphisms
$(R_{\alpha})_*:DH_{\vec{Q}}\rightarrow
DH_{\sigma_{\alpha}\vec{Q}}$ have been given in
\cite{SVdB} and \cite{XY}. Similarly, Burban and Schiffmann have
associated algebra automorphisms of $DH_{\mathcal{A}}$ to
the auto-equivalences of the derived categories of coherent sheaves
on elliptic curves (\cite{BS1}) and weighted projective lines of
tubular type (\cite{BS2}). In \cite{Sc}, Schiffmann states a
conjecture generalizing these formulas to any tilting functor. 

In this paper we consider hereditary categories subject to finiteness conditions which are satisfied by all of the above examples. 

\begin{definition}
We say that an essentially small category $\mathcal{A}$ is \textit{finitary} if $\mbox{Hom}(A,B)$ and $\mbox{Ext}^1(A,B)$ are finite sets for all $A,B\in \mathcal{A},$ and if for some number $k$ there is a homomorphism $d:K_0(\mathcal{A})\rightarrow \mathbb{Z}^k$ such that

i) $d(K^+_0(\mathcal{A}))\subset \mathbb{N}^k\cup (\mathbb{N}^{k-1} \backslash \{0\})\times \mathbb{Z}$ 

ii) If $d([A])=0,$ then $A=0.$

Here $K^+_0(\mathcal{A})$ denotes the subset of $K_0(\mathcal{A})$ corresponding to classes of objects in $\mathcal{A}.$
\end{definition}
This condition is satisfied by the category of representations of a quiver over $\mathbb{F}_q$ (or, equivalently, the category of finite-dimensional modules of a hereditary algebra over $\mathbb{F}_q$) and by the category of coherent sheaves on a smooth projective curve or a weighted projective line over $\mathbb{F}_q$ (see Sections 7-9).

The main result of this paper can be stated as follows

\begin{theorem}

Let $\mathcal{A}$ and $\mathcal{B}$ be two finitary hereditary Abelian categories. If there exists a derived equivalence
$F:\mathcal{D}^b\mathcal{A}\rightarrow \mathcal{D}^b\mathcal{B},$
then $DH_{\mathcal{A}}$ and $DH_{\mathcal{B}}$
are isomorphic as algebras.

\end{theorem}
\emph{Acknowledgements.} I would like to thank Professor Mikhail
Kapranov for his support and many useful discussions. I am also
grateful to Professor Igor Burban for pointing out mistakes in an
earlier version of the paper.

\section{The Drinfeld double}

Given two Hopf algebras $\Gamma$ and $\Lambda,$ a Hopf pairing is a bilinear
map $\varphi: \Gamma \times \Lambda \rightarrow \mathbb{C}$ satisfying for all
$a,a'\in \Gamma$ and $b,b' \in \Lambda$

\begin{equation}
\la{h1}
\varphi(1,b)= \epsilon_{\Lambda}(b), \varphi(a,1)=\epsilon_{\Gamma}(a),
\end{equation}

\begin{equation}
\la{h2} \varphi(a,bb')=\varphi(\triangle_{\Gamma}(a), b\otimes b'),
\end{equation}

\begin{equation}
\la{h3} \varphi(aa',b)=\varphi(a\otimes a', \triangle_{\Lambda}(b)),
\end{equation}

\begin{equation}
\la{h4} \varphi(\sigma_{\Gamma}(a),b)=\varphi(a,\sigma_{\Lambda}(b)),
\end{equation}
where $\epsilon, \triangle,$ and $\sigma$ denote the counit,
coproduct, and antipode, respectively. Here we define
$\varphi:(\Gamma \otimes \Gamma)\times (\Lambda\otimes \Lambda)\rightarrow \mathbb{C}$ by
the rule $\varphi(a\otimes a', b\otimes
b')=\varphi(a,b)\varphi(a',b').$

A skew-Hopf pairing is a bilinear map $\varphi: \Gamma\times \Lambda\rightarrow
\mathbb{C}$ satisfying \eqref{h1}-\eqref{h2} and

\begin{equation}
\la{h33} \varphi(aa',b)=\varphi(a\otimes a', \triangle_{\Lambda}^{opp}(b)),
\end{equation}

\begin{equation}
\la{h44} \varphi(\sigma_{\Gamma}(a),b)=\varphi(a,\sigma_{\Lambda}^{-1}(b)).
\end{equation}

When there exists a skew-Hopf pairing $\varphi:\Gamma\times \Lambda\rightarrow
\mathbb{C},$ the Drinfeld double (\cite{D2}, \cite{J}) of $\Gamma$ and $\Lambda$
with respect to $\varphi$ is the vector space $\Gamma\otimes \Lambda$ with
multiplication defined by

\begin{equation}
\la{d1}
(a\otimes 1)(a'\otimes 1) = aa'\otimes 1,
\end{equation}

\begin{equation}
\la{d2}
(1\otimes b)(1\otimes b')=1\otimes bb',
\end{equation}

\begin{equation}
\la{d3}
(a\otimes 1)(1\otimes b)=a\otimes b,
\end{equation}

\begin{equation}
\la{d4} (1\otimes b)(a\otimes 1)=\sum \varphi(a_{(1)},
\sigma_{\Lambda}(b_{(1)})) a_{(2)}\otimes b_{(2)} \varphi(a_{(3)}, b_{(3)}),
\end{equation}
for all $a,a'\in \Gamma, b, b'\in \Lambda.$ The last identity \eqref{d4} is
equivalent to

\begin{equation}
\la{d5} \sum \varphi(a_{(2)}, b_{(2)}) a_{(1)}\otimes b_{(1)} = \sum
\varphi(a_{(1)}, b_{(1)}) (1\otimes b_{(2)})(a_{(2)}\otimes 1)
\end{equation}
for all $a\in \Gamma, b\in \Lambda.$

The following lemma is adapted from Lemma 3.2 in \cite{BS1}:
\begin{lemma}
For any $a\in \Gamma$ and $b\in \Lambda,$ let $D(a,b)$ denote the equation
\eqref{d5}. Then for $a,a'\in \Gamma$ and $b\in \Lambda,$ the equation
$D(aa',b)$ is implied by the collection of equations $D(a, b_{(1)})$
and $D(a', b_{(2)}).$ Similarly, for $a\in \Gamma$ and $b,b'\in \Lambda,
D(a,bb')$ is implied by the collection of equations $D(a_{(1)}, b)$
and $D(a_{(2)}, b').$

\qed

\end{lemma}

We now mention the original motivation for the Drinfeld double
construction. Let $\mathfrak{g}$ be a Kac-Moody algebra and consider
$U_q(\mathfrak{b}_+),$ the quantized enveloping algebra of a Borel
subalgebra $\mathfrak{b_+}\subset \mathfrak{g}.$ Note that this is a
Hopf algebra. We would like to recover $U_q(\mathfrak{g})$ from
$U_q(\mathfrak{b}_+).$ Let $\Gamma=U_q(\mathfrak{b}_+)$ and let
$\Lambda=\Gamma^{coop},$ i.e. $\Gamma$ with opposite coproduct. Then there exists a
symmetric skew-Hopf pairing $\varphi:\Gamma\times \Lambda\rightarrow
\mathbb{C}$ defined by
\[\varphi(E_i, E_j)=\frac{\delta_{i,j}}{q-1}, \mbox{   }
\varphi(K_i,K_j)=q^{(i,j)/2}, \mbox{   } \varphi(E_i,K_j)=0.\] The
quotient of the Drinfeld double of $\Gamma$ and $\Lambda$ with respect to this
pairing by the ideal generated by the elements $K_i\otimes
1-1\otimes K_i^{-1}$ is isomorphic to $U_q(\mathfrak{g}).$

\section{Hall algebras}
Assume $\mathcal{A}$ is an essentially small Abelian category of
finite global dimension with Ext$^i(M,N)$ finite for all $i.$ Then
we can define a new multiplication on $\mathbb{C}[\mathcal{I}]$ that
is still associative by
\[[A]*[B]=\frac{1}{\langle B,A \rangle}\sum_{[C]\in \mathcal{I}} g_{AB}^C [C],\]
where the Euler form $\langle \cdot, \cdot
\rangle:K_0(\mathcal{A})\times K_0(\mathcal{A})\rightarrow
\mathbb{C}$ is defined by
\[\langle M,N \rangle = \sqrt{\Pi_{i\ge 0}^{\infty}
|\mbox{Ext}^{i}(M,N)|^{(-1)^i}}.\] We will call this the Ringel-Hall
algebra of $\mathcal{A}$ and denote it by $H_{\mathcal{A}}.$ In the
case of $\mathcal{A}=Rep_{\mathbb{F}_q}(\vec{Q}),$ the Ringel-Hall
algebra of $\mathcal{A}$ (or its ``composition subalgebra" if
$\vec{Q}$ is not finite type) also gives a realization of
$U_q(\mathfrak{n}_+).$

Next, we define the extended Ringel-Hall algebra
$\tilde{H}_{\mathcal{A}}$ to be the associative algebra generated by
$H_{\mathcal{A}}$ and $\mathbb{C}[K_0(\mathcal{A})]$ subject to the
relation
\[k_{\alpha} [M] = ( \alpha | M ) [M] k_{\alpha},\] where $(A|B):=\langle A,B \rangle \langle B,A
\rangle.$ Here $k_{\alpha}$ represents a basis element of
$\mathbb{C}[K_0(\mathcal{A})]$ corresponding to $\alpha\in
K_0(\mathcal{A}).$ There is an isomorphism of vector spaces given by
the multiplication
\[m:\mathbb{C}[K_0(\mathcal{A})]\otimes H_{\mathcal{A}}\rightarrow
\tilde{H}_{\mathcal{A}}.\] Note that $[0]\in \mathcal{I}$ is the
unit in both $H_{\mathcal{A}}$ and $\tilde{H}_{\mathcal{A}}.$

If we assume that $\mathcal{A}$ is hereditary and that each object
of $\mathcal{A}$ has finitely many subobjects, then
$\tilde{H}_{\mathcal{A}}$ is a self-dual Hopf algebra, as shown by
Xiao \cite{X2}. In the case that $\vec{Q}$ is an A-D-E quiver, the extended Hall algebra $\tilde{H}_{\mathcal{A}}$ of $\mathcal{A}=Rep_{\mathbb{F}_q}(\vec{Q})$ gives a realization of $U_q(\mathfrak{b}_+)$ as a Hopf algebra. In general, the coproduct, counit, and antipode for $\tilde{H}_{\mathcal{A}}$ are
given by (\cite{Gr}, \cite{X2}):

\begin{equation} \la{cop1}
\triangle([A])=\sum_{[A'],[A'']\in \mathcal{I}} \langle A',A''
\rangle \frac{|\Delta_{A'',A'}^A|}{|\mbox{Aut} A|}
k_{A''}[A']\otimes[A''],
\end{equation}
\begin{equation} \la{cop2}
\triangle(k_{\alpha})=k_{\alpha}\otimes k_{\alpha},
\end{equation}
\begin{equation} \la{cou}
\epsilon([A]k_{\alpha})=\delta_{[A],[0]},
\end{equation}
\begin{equation} \la{antip}
\sigma([A])=\frac{k_A^{-1}}{|\mbox{Aut}A|} \left(\sum_{n=1} (-1)^n \sum_{L_{\bullet}\in F_{A,n}} \prod_{i=1}^n \langle L_i/L_{i+1}, L_{i+1} \rangle (L_i/L_{i+1} | L_{i+1}) |\mbox{Aut} L_i/L_{i+1}| [L_i/L_{i+1}] \right),
\end{equation}
\begin{equation} \la{antip2}
\sigma(k_{\alpha})=k_{\alpha}^{-1},
\end{equation}
for $[A]\in \mathcal{I}$ and $\alpha\in K_0(\mathcal{A}).$ Here the indexing set $F_{A,n}$ denotes the set of all strict $n$-step filtrations
\[0=L_{n+1}\subset L_n \subset \cdots \subset L_2 \subset L_1=A.\] Note that $\tilde{H}_{\mathcal{A}}$ is graded as a Hopf algebra by the Grothendieck group $K_0(\mathcal{A}).$

We will use a pairing between $\tilde{H}_{\mathcal{A}}$ and $\tilde{H}^{coop}_{\mathcal{A}}$ which is a renormalization of Green's scalar product on $\tilde{H}_{\mathcal{A}}.$

\begin{proposition}
The bilinear map $\varphi:\tilde{H}_{\mathcal{A}}\times
\tilde{H}^{coop}_{\mathcal{A}}\rightarrow \mathbb{C}$ defined by

\[\varphi(k_{\alpha}[M], k_{\beta} [M']) = (\alpha| M)(M| M)(M|
\beta)(\alpha| \beta) \frac{\delta_{[M],[M']}}{|\mbox{Aut} M|}\]

is a skew-Hopf pairing.
\end{proposition}

\begin{proof}
It is clear that \eqref{h1} holds. We check \eqref{h2} and
\eqref{h33} follows from a similar argument.

\begin{eqnarray} \nonumber
\varphi(k_{\alpha}[A],
k_{\beta}[B]k_{\gamma}[C])&=&\frac{1}{(B|\gamma)}\varphi(k_{\alpha}[A],
k_{\beta+\gamma}[B][C])\\ \nonumber
&=&\frac{(\alpha|\beta+\gamma)(A|\beta+\gamma)}{(B|\gamma)}\varphi(k_{\alpha}[A],
[B][C])\\ \nonumber &=&
\frac{(\alpha|\beta+\gamma)(A|\beta+\gamma)}{(B|\gamma)\langle C,B
\rangle}\varphi(k_{\alpha}[A], \sum_{[D]} g_{B,C}^D [D])\\ \nonumber
&=& \frac{(\alpha|\beta+\gamma)(A|\beta+\gamma)}{(B|\gamma)\langle
C,B \rangle} \varphi(k_{\alpha}[A], g_{B,C}^A [A])\\ \nonumber &=&
\frac{(\alpha|\beta+\gamma)(A|\beta+\gamma)(\alpha|A)(A|A)g_{B,C}^A}{(B|\gamma)\langle
C,B \rangle |\mbox{Aut} A|} \\ \nonumber &=&
\frac{(\alpha|\beta+\gamma)(B+C|\beta+\gamma)(\alpha|B+C)(B+C|B+C)g_{B,C}^A}{(B|\gamma)\langle
C,B \rangle |\mbox{Aut} A|} \\ \nonumber &=&
\frac{(\alpha|\beta)(\alpha|\gamma)(B|\beta)(B|\gamma)(C|\beta)(C|\gamma)(\alpha|B)(\alpha|C)(B|B)(B|C)(B|C)(C|C)g_{B,C}^A}{(B|\gamma)\langle
C,B \rangle |\mbox{Aut} A|} \\ \nonumber &=&
\frac{(\alpha|\beta)(\alpha|\gamma)(B|\beta)(C|\beta)(C|\gamma)(\alpha|B)(\alpha|C)(B|B)(B|C)(C|C)\langle
B,C \rangle g_{B,C}^A}{|\mbox{Aut} A|}
\end{eqnarray}

\begin{eqnarray} \nonumber
\varphi(\triangle(k_{\alpha}[A]), k_{\beta}[B]\otimes k_{\gamma}[C])
&=& \varphi(\sum_{[A'],[A'']} \langle A', A'' \rangle
\frac{|\Delta_{A'',A'}^A|}{|\mbox{Aut}A|} k_{\alpha+A''}[A']\otimes
k_{\alpha}[A''], k_{\beta}[B]\otimes k_{\gamma}[C]) \\ \nonumber &=&
\varphi(\langle B, C \rangle
\frac{|\Delta_{C,B}^A|}{|\mbox{Aut}A|}k_{\alpha+C}[B]\otimes
k_{\alpha}[C], k_{\beta}[B]\otimes k_{\gamma}[C])\\ \nonumber &=&
\frac{\langle B, C \rangle |\Delta_{C,B}^A|
(\alpha+C|\beta)(\alpha+C|B)(B|\beta)(B|B)(\alpha|\gamma)(\alpha|C)(C|\gamma)(C|C)}{|\mbox{Aut}A||\mbox{Aut}B||\mbox{Aut}C|}\\
\nonumber &=& \frac{g_{B,C}^A \langle B, C \rangle (\alpha|\beta)(C|\beta)(\alpha|B)(C|B)(B|\beta)(B|B)(\alpha|\gamma)(\alpha|C)(C|\gamma)(C|C)}{|\mbox{Aut}A|}\\
\nonumber
\end{eqnarray}
Together, $\varphi(k_{\alpha}[A],
k_{\beta}[B]k_{\gamma}[C])=\varphi(\triangle(k_{\alpha}[A]),
k_{\beta}[B]\otimes k_{\gamma}[C]),$ verifying \eqref{h2}. Together, the conditions \eqref{h1}, \eqref{h2}, and
\eqref{h33} imply \eqref{h44}, and we are done.

\end{proof}

\section{Double Hall algebras}

In this section we let $\mathcal{A}$ be a finitary hereditary Abelian category. Without the requirement that every object in $\mathcal{A}$ has finitely many subobjects, the sums \eqref{cop1} and \eqref{antip} may be infinite. However, there are weaker conditions satisfied by finitary categories that will allow us to apply the Drinfeld double construction to $\tilde{H}_{\mathcal{A}}$. The following definition and lemma are adapted from Definition B.2 and Lemma B.3 in \cite{BS1}.

\begin{definition}

Given two objects $A,B\in \mathcal{A},$ an \textit{anti-equivalence} between $A$ and $B$ is a pair of strict filtrations
\[(0=L_{n+1}\subset L_n\subset \cdots \subset L_1 \subset L_0=A, 0=M_{n+1}\subset M_n\subset \cdots \subset M_1 \subset M_0=B)\]
such that $L_i/L_{i+1} \simeq M_{n-i}/M_{n-i+1}$ for all $i.$ We say that two objects $A$ and $B$ are \textit{anti-equivalent} if there exists at least one anti-equivalence between them.

\end{definition}

\begin{lemma}

1) Given any two objects $A,B\in \mathcal{A},$ there exist finitely many anti-equivalences between $A$ and $B.$

2) Let $A$ and $B$ be any two objects in $\mathcal{A}.$ Then there are finitely many pairs of anti-equivalent objects $(A',B')$ such that $A'$ is a subobject of $A$ and $B'$ is a quotient of $B.$

\end{lemma}

\begin{proof}
We define an order on $K^+_0(\mathcal{A})\times K^+_0(\mathcal{A})$ by setting $d_2(A,B)=\min(d(A),d(B)),$ where $d:K_0(\mathcal{A})\rightarrow \mathbb{Z}^k$ is the homomorphism in Definition 1 and $\mathbb{Z}^k$ is ordered lexicographically. We prove both parts of the lemma using induction on the pair $(A,B),$ with the case of $d_2(A,B)\in (0,\cdots, 0, 1)$ being trivial in each case.

Let $\Sigma(A,B)$ denote the set of anti-equivalences between $A$ and $B.$  Let $S(A,B)$ denote the set of objects in $\mathcal{A}$ that are both subobjects of $A$ and quotients of $B.$ Since $|\mbox{Hom}(B,A)|<\infty,$ the set $S(A,B)$ is finite. If
\[(0=L_{n+1}\subset L_n\subset \cdots \subset L_1 \subset L_0=A, 0=M_{n+1}\subset M_n\subset \cdots \subset M_1 \subset M_0=B)\]
is an anti-equivalence between $A$ and $B,$ then $L_n\in S(A,B).$ For an object $L\in S(A,B),$ let $\Sigma_L(A,B)$ denote the set of anti-equivalences between $A$ and $B$ such that $L$ is the first non-zero term in the corresponding filtration of $A.$ Then we have
\begin{equation} \la{antieq}
\Sigma(A,B)=\bigcup_{L\in S(A,B)} \Sigma_L(A,B).
\end{equation}
Given an anti-equivalence between $A$ and $B$ of the form
\[(0=L_{n+1}\subset L_n\subset \cdots \subset L_1 \subset L_0=A, 0=M_{n+1}\subset M_n\subset \cdots \subset M_1 \subset M_0=B),\]
we can associate a pair of filtrations
\[(0 \subset L_{n-1}/L_n \subset \cdots \subset L_1/L_n \subset A/L_n, 0=M_{n+1}\subset M_n\subset \cdots \subset M_1)\]
that is an anti-equivalence between $A/L_n$ and $M_1.$ In fact, this establishes a bijection between the sets $\Sigma_{L_n}(A,B)$ and $\Sigma(A/L_n, M_1).$ Because the union in \eqref{antieq} is taken over the finite set $S(A,B),$ the first part of the lemma follows by induction.

Now suppose that $A'$ is a subobject of $A$ and $B'$ is a quotient of $B.$ Then it can also be shown that
\[\Sigma(A',B')=\bigcup_{L\in S(A,B)} \Sigma_L(A',B').\]
So to prove the second part of the lemma, it suffices to show that for any $L\in S(A,B),$ there are finitely many pairs of objects $(A',B')$ such that 

(a) $A'$ is a subobject of $A,$ 

(b) $B'$ is a quotient of $B,$ and 

(c) $\Sigma_L(A',B')$ is non-empty.

If $(A',B')$ satisfies the conditions (a) and (b), then $A'/L$ is a subobject of $A/L$ and $\mathrm{ker}(B'\twoheadrightarrow L)$ is a quotient of $\mathrm{ker}(B\twoheadrightarrow L).$ Moreover, if the pair $(A',B')$ satisfies condition (c), then $\Sigma(A'/L,\mathrm{ker}(B'\twoheadrightarrow L))$ is non-empty and the objects $A'/L$ and $\mathrm{ker}(B'\twoheadrightarrow L)$ are anti-equivalent. By our induction hypothesis, this says that there are only a finite number of values that $A'/L$ and $\mathrm{ker}(B'\twoheadrightarrow L)$ can take if $A'$ and $B'$ satisfy the three conditions above. Because $|\mbox{Ext}^1(M,N)|<\infty$ for any $M,N\in \mathcal{A},$ it follows that there are finitely many pairs of objects $(A',B')$ satisfying (a)-(c).

\end{proof}

We define the inverse antipode $\sigma^{-1}:\tilde{H}_{\mathcal{A}}\rightarrow \tilde{H}_{\mathcal{A}}$ by

\begin{equation} \la{invant1}
\sigma^{-1}([A])=\frac{1}{|\mbox{Aut}A|} \left( \sum_{n=1} (-1)^n \sum_{L_{\bullet}\in F_{A,n}} \prod_{i=1}^n \langle L_i/L_{i+1}, L_{i+1} \rangle |\mbox{Aut} L_i/L_{i+1}| [L_{n -i+1}/L_{n-i+2}] \right) k_A^{-1}
\end{equation}
and
\begin{equation} \la{invantk}
\sigma^{-1}(k_{\alpha})=k_{\alpha}^{-1}.
\end{equation}
Although the sum in \eqref{invant1} may be infinite as well, the coefficient of a given basis element $[B]$ will be finite by Lemma 2.

Now, using \eqref{h44} to replace $\sigma$ with $\sigma^{-1},$ one can show that the relations \eqref{d1}-\eqref{d4} define an associative multiplication on the tensor product $\tilde{H}_{\mathcal{A}}\otimes \tilde{H}^{coop}_{\mathcal{A}}$ (equivalently, $\tilde{H}_{\mathcal{A}}\otimes \tilde{H}_{\mathcal{A}}$) by similar arguments as in the case when $\tilde{H}_{\mathcal{A}}$ is a Hopf algebra (see e.g. \cite{J}). The only essential difference is that we must first show that the sums \eqref{d4} and \eqref{d5} are finite.

\begin{proposition}
Let $\Gamma=\tilde{H}_{\mathcal{A}}$ and $\Lambda=\tilde{H}^{coop}_{\mathcal{A}}.$ Then the relations \eqref{d4} and \eqref{d5} consist of finite sums for any elements $a\in \Gamma, b\in \Lambda.$
\end{proposition}

\begin{proof}
First, we rewrite the sum \eqref{cop1} as

\[\triangle([A]) = \sum_{L \subset A} \langle A/L, L \rangle \frac{|\mbox{Aut} A/L||\mbox{Aut} L|}{|\mbox{Aut} A|} k_L [A/L]\otimes [L]\]
for $A\in \mathcal{A}$. By coassociativity, we have

\[\triangle^2 ([A]) = \sum_{L_2 \subset L_1 \subset A} \langle A/L_2, L_2 \rangle \langle L_1/L_2, L_2 \rangle \frac{|\mbox{Aut} A/L_1||\mbox{Aut} L_1/L_2||\mbox{Aut} L_2|}{|\mbox{Aut} A|} k_{L_1} [A/L_1] \otimes k_{L_1-L_2} [L_1/L_2] \otimes [L_2]. \]

If we set $a=k_{\alpha}[A]$ and $b=k_{\beta}[B],$ the right hand side of \eqref{d4} is a sum of terms

\begin{equation} \la{rhs}
C_{L_{\bullet},M_{\bullet},\alpha, \beta} \cdot \varphi(\sigma^{-1}([A/L_1]), [M_2]) \cdot \varphi([L_2], [B/M_1]) \cdot k_{L_1-L_2}[L_1/L_2]\otimes k_{M_1-M_2}[M_1/M_2]
\end{equation}
over pairs of filtrations $(L_2 \subset L_1 \subset A, M_2 \subset M_1 \subset B).$ Here $C_{L_{\bullet},M_{\bullet},\alpha, \beta}$ is a complex number depending on $\alpha, \beta,$ and the pair of filtrations. Because $\mbox{Hom}(B,A)$ is finite, the term 
\[\varphi([L_2], [B/M_1])\]
is $0$ for all but a finite number of choices of $L_2$ and $M_1.$ Similarly, by Lemma 2, the term
\[\varphi(\sigma^{-1}([A/L_1]), [M_2])\]
will be $0$ for all but a finite number of $L_1$ and $M_2.$ It follows that the right hand side of \eqref{d4} is a finite sum in this case. Since $\tilde{H}_{\mathcal{A}}$ and $\tilde{H}^{coop}_{\mathcal{A}}$ are each spanned by elements of the form $k_{\alpha}[A]$ and $k_{\beta}[B],$ we can conclude that \eqref{d4} consists of finite sums for all $a\in \Gamma$ and $b\in \Lambda.$ A similar argument shows that the sums in \eqref{d5} are finite.

\end{proof}

We can now look at the algebra structure on $\tilde{H}_{\mathcal{A}}\otimes \tilde{H}_{\mathcal{A}}$ in some detail. Substituting $a=k_{\alpha}[A]$ and $b=k_{\beta}[B]$ into \eqref{d5} gives us

\begin{equation} \la{eq1}
\sum_{[L],[M],[N]} \frac{\langle A, B \rangle \langle B, B
\rangle}{\langle L, N \rangle \langle N, N \rangle}
\frac{|\Delta_{M,L}^A||\Delta_{N,M}^B|}{|\mbox{Aut} A| |\mbox{Aut}
B| |\mbox{Aut} M|} (k_M [L]\otimes 1)(1\otimes [N]) =
\end{equation}

\begin{equation} \la{eq2}
\sum_{[L],[M],[N]} \frac{\langle B, A \rangle \langle A, A
\rangle}{\langle N, L \rangle \langle L, L \rangle}
\frac{|\Delta_{L,M}^A||\Delta_{M,N}^B|}{|\mbox{Aut} A| |\mbox{Aut}
B| |\mbox{Aut} M|}(1\otimes k_M[N])([L]\otimes 1).
\end{equation}

Note that this identity is independent of $k_{\alpha}$ and
$k_{\beta}.$ For any $A,B,L,N\in \mathcal{A},$ the four-term exact sequence

\[0\rightarrow N\rightarrow B\rightarrow A\rightarrow L\rightarrow
0\] has a unique decomposition into two short exact sequences

\[0\rightarrow N\rightarrow B\rightarrow M\rightarrow 0\]
and
\[0\rightarrow M\rightarrow A\rightarrow L\rightarrow 0,\]
and conversely, any pair of short exact sequences of this form
determines a unique four-term exact sequence. It follows that for fixed $[L], [N] \in \mathcal{I},$

\begin{equation} \la{four}
\sum_{[M]}
\frac{|\Delta_{M,L}^A||\Delta_{N,M}^B|}{|\mbox{Aut}A||\mbox{Aut}B||\mbox{Aut}
M|}=\frac{|\{\phi:B\rightarrow A|\mbox{Ker}(\phi)\simeq N,
\mbox{Coker}(\phi)\simeq
L\}||\mbox{Aut}N||\mbox{Aut}L|}{|\mbox{Aut}A||\mbox{Aut}B|}.
\end{equation}

It will be useful for us to rewrite \eqref{eq1}-\eqref{eq2} in terms of exact triangles. Define

\[g_{A_{\bullet},
B_{\bullet}}^{C_{\bullet}}=\frac{|\Delta_{B_{\bullet},
A_{\bullet}}^{C_{\bullet}}|}{|\mbox{Aut} A_{\bullet}||\mbox{Aut}
B_{\bullet}|},\] where $\Delta_{B_{\bullet},
A_{\bullet}}^{C_{\bullet}}$ denotes the set of exact triangles

\[B_{\bullet}\rightarrow C_{\bullet} \rightarrow A_{\bullet}
\rightarrow B_{\bullet}[1]\] in $D^b(\mathcal{A}),$ the bounded
derived category of $\mathcal{A}$ (\cite{GM}). This set is finite because $\mbox{Hom}(M,N)$ and $\mbox{Ext}^1(M,N)$ are finite sets for any $M,N\in \mathcal{A}.$ We adapt the
following lemma from Proposition 2.4.3 in \cite{K2}:

\begin{lemma}

For any objects $A,B,L,N \in \mathcal{A},$

\begin{equation} \nonumber
\sum_{[M]} \frac{|\Delta_{M,L}^A||\Delta_{N,M}^B|}{|Aut A||Aut
B||Aut M|}=\frac{g_{B[1],A}^{N[1]\oplus L}}{|Ext^1(L,N)|}
\end{equation}

\end{lemma}

\qed

As a result of the lemma, \eqref{eq1}-\eqref{eq2} becomes

\begin{equation}
\la{eq3}
 \sum_{[L],[N]} \frac{\langle A, B \rangle \langle B, B
\rangle}{\langle L, N \rangle \langle N, N \rangle}
\frac{g_{B[1],A}^{N[1]\oplus L}}{|\mbox{Ext}^1(L,N)|} (k_{A-L}
[L]\otimes 1)(1\otimes [N]) =
\end{equation}

\begin{equation}
\la{eq4} \sum_{[L],[N]} \frac{\langle B, A \rangle \langle A, A
\rangle}{\langle N, L \rangle \langle L, L \rangle}
\frac{g_{A[1],B}^{L[1]\oplus N}}{|\mbox{Ext}^1(N,L)|}(1\otimes
k_{A-L}[N])([L]\otimes 1).
\end{equation}

The following definition is motivated by the construction of the quantum group $U_q(\mathfrak{g})$ from $U_q(\mathfrak{b}_{+}).$

\begin{definition}
The ``double Hall algebra" $DH_{\mathcal{A}}$ is the quotient of $\tilde{H}_{\mathcal{A}}\otimes \tilde{H}_{\mathcal{A}}$ by the ideal generated by the elements $k_{\alpha}\otimes 1 - 1\otimes k_{\alpha}^{-1}.$
\end{definition}

\section{Derived equivalences and normal form}
The following result can be found in \cite{K2}:

\begin{proposition}
Suppose that $\mathcal{A}$ is hereditary. Then each indecomposable
object of $\mathcal{D}^b(\mathcal{A})$ is of the form $A[i]$ where
$A$ is an indecomposable object in $\mathcal{A}.$

\qed

\end{proposition}

Let $F:\mathcal{D}^b(\mathcal{A})\rightarrow
\mathcal{D}^b(\mathcal{B})$ be a derived equivalence. Define
$\mathcal{A}_i$ to be the full subcategory of $\mathcal{A}$ with
objects $\{A\in \mathcal{A}|F(A)\subset \mathcal{B}[i]\}$ and define
$\mathcal{B}_i=F(\mathcal{A}_i)[-i].$ Then for $A_i\in
\mathcal{A}_i, A_j\in \mathcal{A}_j, B_i\in \mathcal{B}_i, B_j\in
\mathcal{B}_j,$ we have

\vspace{1mm}

\[\mbox{Hom}_{\mathcal{A}}(A_i, A_j)=\mbox{Hom}_{\mathcal{D}^b(\mathcal{B})}(B_i[i], B_j[j])=\mbox{Ext}^{j-i}_{\mathcal{B}}(B_i,B_j)=0, \mbox{  unless } j=i \mbox{ or } j=i+1\]

\[\mbox{Hom}_{\mathcal{B}}(B_i,B_j)=\mbox{Hom}_{\mathcal{D}^b(\mathcal{A})}(A_i[-i],A_j[-j])=\mbox{Ext}_{\mathcal{A}}^{i-j}(A_i,A_j)=0, \mbox{  unless } j=i \mbox{ or } j=i-1\]

\begin{eqnarray} \nonumber
\mbox{Ext}^1_{\mathcal{A}}(A_i,A_j) =
\mbox{Hom}_{\mathcal{D}^b(\mathcal{A})}(A_i,
A_j[1])&=&\mbox{Hom}_{\mathcal{D}^b(\mathcal{B})}(B_i[i],B_j[j+1])\\
\nonumber &=&\mbox{Ext}_{\mathcal{B}}^{j-i+1}(B_i,B_j) = 0, \mbox{
unless } j=i \mbox{ or } j=i-1 \nonumber
\end{eqnarray}

\begin{eqnarray} \nonumber
\mbox{Ext}_{\mathcal{B}}^1(B_i,B_j)=\mbox{Hom}_{\mathcal{D}^b(\mathcal{B})}(B_i,B_j[1])&=&\mbox{Hom}_{\mathcal{D}^b(\mathcal{A})}(A_{i}[-i],A_j[-j+1])\\
\nonumber &=&\mbox{Ext}_{\mathcal{A}}^{i-j+1}(A_i,A_j)=0, \nonumber
\mbox{ unless } j=i \mbox{ or } j=i+1 \nonumber
\end{eqnarray}
If we assume that $\mathcal{A}$ and $\mathcal{B}$ satisfy the
hypotheses of Theorem 1, then it follows that $[A_i\oplus A_j]=
\langle A_j, A_i \rangle [A_i]*[A_j]$ whenever $i<j$ and $A_i\in
\mathcal{A}_i, A_j\in \mathcal{A}_j.$ More generally, any object
$A\in \mathcal{A}$ decomposes as a direct sum $\bigoplus_{i\in S}
A_i$ with $A_i\in \mathcal{A}_i,$ and $[A]$ can be written in normal
form
\[[A]=\prod_{i<j} \langle A_j, A_i \rangle \prod_{i \in S} [A_i],\]
where the indices in the product on the right are in increasing
order. This says that the multiplication map

\begin{equation} \nonumber
m:\bigotimes_{i\in \mathbb{Z}} H_{\mathcal{A}_i}\rightarrow
H_{\mathcal{A}}
\end{equation}
is an isomorphism of vector spaces. Here $H_{\mathcal{A}_i}$ denotes
the subalgebra of $H_{\mathcal{A}}$ generated by elements
$\{[A]|A\in \mathcal{A}_i\}$ (or equivalently, the vector subspace
spanned by the elements $\{[A]|A\in \mathcal{A}_i\}$). Note that we
define the infinite tensor product of algebras as spanned by tensors with all but finitely many factors equal to 1. Moreover, the
multiplication

\begin{equation}
\la{dec2} m: \bigotimes_{i\in \mathbb{Z}}
\mathbb{C}[K_0(\mathcal{A}_i)]\otimes \bigotimes_{i\in \mathbb{Z}}
H_{\mathcal{A}_i} \rightarrow \tilde{H}_{\mathcal{A}}
\end{equation}
is a surjective homomorphism, where $\mathbb{C}[K_0(\mathcal{A}_i)]$
denotes the subalgebra of $\mathbb{C}[K_0(\mathcal{A})]$ generated
by the elements $k_A$ for $A\in \mathcal{A}_i.$

\begin{proposition}

Let $F:D^b(\mathcal{A})\rightarrow D^b(\mathcal{B})$ be a derived
equivalence and define the subcategories $\mathcal{A}_i$ for $i\in
\mathbb{Z}$ as above. The double Hall algebra
$DH_{\mathcal{A}}$ is isomorphic to the free associative
algebra on the vector space $\tilde{H}_{\mathcal{A}} \otimes
\tilde{H}_{\mathcal{A}}$ modulo the relations

\begin{equation}
([A]\otimes 1)([B]\otimes 1)=\frac{1}{\langle B, A \rangle}
\sum_{[C]\in \mathcal{I}} g_{A,B}^C ([C]\otimes 1)
\end{equation}

\begin{equation}
(k_{A} \otimes 1)(k_{B} \otimes 1)=(k_{A+B}\otimes 1)
\end{equation}

\begin{equation}
(k_{A}\otimes 1)([B]\otimes 1)=(k_{A}[B]\otimes 1)=(A|B)([B]\otimes
1)(k_{A}\otimes 1)
\end{equation}

\begin{equation}
(1 \otimes [A])(1 \otimes [B])=\frac{1}{\langle B, A \rangle}
\sum_{[C]\in \mathcal{I}} g_{A,B}^C (1\otimes [C])
\end{equation}

\begin{equation}
(1\otimes k_{A})(1\otimes k_{B})=(1\otimes k_{A+B})
\end{equation}

\begin{equation}
(1\otimes [A])(1\otimes k_{B})=(1\otimes [A]k_{B})=(A|B)(1\otimes
k_{B})(1\otimes [A])
\end{equation}

\begin{equation}
([A]\otimes 1)(1\otimes [B])=([A]\otimes [B])
\end{equation}

\begin{equation}
k_{A}\otimes 1 = 1\otimes k_{A}^{-1}
\end{equation}

\begin{equation}
\sum_{[L],[N]} \frac{\langle A, B \rangle \langle B, B
\rangle}{\langle L, N \rangle \langle N, N \rangle}
\frac{g_{B[1],A}^{N[1]\oplus L}}{|\mbox{Ext}^1(L,N)|} (k_{A-L}
[L]\otimes 1)(1\otimes [N]) =
\end{equation}

\begin{equation}
\sum_{[L],[N]} \frac{\langle B, A \rangle \langle A, A
\rangle}{\langle N, L \rangle \langle L, L \rangle}
\frac{g_{A[1],B}^{L[1]\oplus N}}{|\mbox{Ext}^1(N,L)|}(1\otimes
k_{A-L}[N])([L]\otimes 1).
\end{equation}

for $A\in \mathcal{A}_i, B\in \mathcal{A}_j, i,j\in \mathbb{Z}.$

\end{proposition}

\begin{proof}
Our original definition of $DH_{\mathcal{A}}$ is clearly
isomorphic to the free associative algebra on
$\tilde{H}_{\mathcal{A}}\otimes \tilde{H}_{\mathcal{A}}$ modulo the
relations (5.2)-(5.11) for $A,B\in \mathcal{A}.$ Using \eqref{dec2},
it is also not hard to see that if the relations (5.2)-(5.9) hold
for all $A\in \mathcal{A}_i$ and $B\in \mathcal{A}_j,$ then they
hold for all $A, B\in \mathcal{A}.$

Assume that (5.10)-(5.11) holds for all $A\in \mathcal{A}_i$ and $B\in \mathcal{A}_j$ for all $i,j\in \mathbb{Z}.$  It is clear that we can use Lemma 1 to show that (5.10)-(5.11) also holds for all $A,B\in \mathcal{A}$ with either $A$ or $B$ simple. Recall the homomorphism $d_2:K_0(\mathcal{A})\times K_0(\mathcal{A})\rightarrow \mathbb{Z}^k$ and the order on $K^+_0(\mathcal{A})\times K^+_0(\mathcal{A})$ defined in the proof of Lemma 2. If $d_2(A,B)=(0,0,\cdots, 1),$ then either $A$ or $B$ is a simple object and the relation (5.10)-(5.11) holds. Therefore, using Lemma 1, we can show that (5.10)-(5.11) holds for all $A,B\in \mathcal{A}$ by induction.

\end{proof}

\section{Proof of Theorem 1}
We now prove the main result that the double Hall algebra is invariant under derived equivalences. Assume
throughout this section that $\mathcal{A}$ and $\mathcal{B}$ satisfy
the hypotheses of Theorem 1.

\begin{proposition}

Let $F:D^b(\mathcal{A})\rightarrow D^b(\mathcal{B})$ be a derived
equivalence. Then the assignment

\[1\otimes [M] \mapsto \langle N, N \rangle ^n (1\otimes
[N]k_N^n) \mbox{ for } n \mbox{ even}\]

\[1\otimes [M] \mapsto \langle N, N \rangle ^n
([N]k_N^n\otimes 1) \mbox{ for } n \mbox{ odd }\]

\[k_M\otimes 1 \mapsto k_N\otimes 1 \mbox{ for } n \mbox{ even}\]

\[k_M\otimes 1 \mapsto 1\otimes k_N \mbox{ for } n \mbox{ odd}\]

\[[M]\otimes 1 \mapsto \langle N, N \rangle ^n ([N]k_N^n \otimes 1) \mbox{ for } n \mbox{ even}\]

\[[M]\otimes 1 \mapsto \langle N, N \rangle ^n
(1\otimes [N]k_N^n) \mbox{ for } n \mbox{ odd }\]

\vspace{3mm}

where $M\in \mathcal{A}_n$ and $N= F(M)[-n],$ extends to a
homomorphism $F_*:DH_{\mathcal{A}}\rightarrow
DH_{\mathcal{B}}.$

\end{proposition}

\begin{proof}
First, we can extend $F_*$ to a linear map
$F_*:\tilde{H}_{\mathcal{A}}\otimes
\tilde{H}_{\mathcal{A}}\rightarrow DH_{\mathcal{B}}.$ For
$M\in \mathcal{A}_n$ and $N=F(M)[-n],$ define
\[F_*(1\otimes k_M)=1 \otimes k_N\]
for $n$ even and
\[F_*(1\otimes k_M)=k_N\otimes 1\]
for $n$ odd. By \eqref{dec2}, we can decompose $k_{\alpha}[A]$
uniquely as
\[k_{\alpha}[A]=\prod_i k_{\alpha_i} \prod_{i<j} \langle A_j, A_i \rangle
\prod_i [A_i]\] for any $\alpha\in K_0(\mathcal{A})$ and $A\in
\mathcal{A}.$ So, define
\[F_*(k_{\alpha}[A]\otimes 1)= \prod_i F_*(k_{\alpha_i}\otimes 1) \prod_{i<j} \langle A_j, A_i \rangle
\prod_i F_*([A_i]\otimes 1).\] Define $F_*(1\otimes k_{\beta}[B])$
similarly, and define
\[F_*(k_{\alpha}[A]\otimes k_{\beta}[B])=F_*(k_{\alpha}[A]\otimes
1)F_*(1\otimes k_{\beta}[B]).\] Then $F_*$ extends to a linear map
on $\tilde{H}_{\mathcal{A}}\otimes \tilde{H}_{\mathcal{A}}$ and to
an algebra homomorphism on the free associative algebra on
$\tilde{H}_{\mathcal{A}}\otimes \tilde{H}_{\mathcal{A}}.$ By
Proposition 3, it suffices to show that $F_*$ preserves
(5.2)-(5.11).

It is immediate from our definition of $F_*$ that (5.8) and (5.9)
are preserved for $A\in \mathcal{A}_i$ and $B\in \mathcal{A}_j.$ It
is also easy to see that (5.3) and (5.6) are preserved for $A\in
\mathcal{A}_i$ and $B\in \mathcal{A}_j.$ This is because
$F$ induces a group homomorphism on $K_0(\mathcal{A}).$

Now, take $A\in \mathcal{A}_i$ and $A'\in \mathcal{A}_{i+n}.$
Without loss of generality, we may assume that $i$ is even. If we
set $B=F(A)[-i]$ and $B'=F(A')[-i-1],$ then for $n$ even,

\[F_*(k_A\otimes 1)F_*([A']\otimes 1) = F_*(k_A[A']\otimes 1) = \langle B, B \rangle^i k_B
[B']k_{B'}^i \otimes 1=\]
\[\langle  B, B \rangle^i (B|B')[B']k_B k_{B'}^i \otimes 1 = (B|B')F_*([A']\otimes 1)F_*(k_A\otimes
1),\] and for $n$ odd,

\[F_*(k_A\otimes 1)F_*([A']\otimes 1) = F_*(k_A[A']\otimes 1) = \langle B, B \rangle^i k_B
\otimes [B']k_{B'}^i=\]

\[\frac{\langle B, B \rangle^i}{(B|B')}(1\otimes
[B']k_B^{-1}k_{B'}^i)=\frac{1}{(B|B')}F_*([A']\otimes
1)F_*(k_A\otimes 1).\] Note that $(A|A')=(B|B')$ if $n=0$ and
$(A|A')=(B|B')^{-1}$ if $n=1.$ If $n\ge 2,$ then $(A|A')=(B|B')=1.$
Hence, $F_*$ is consistent with (5.4), and by a similar argument, it
is consistent with (5.7).

We divide the proof that $F_*$ preserves (5.5) into three cases. The
proof for (5.2) is analogous.

\underline{Case 1: $i=j$}

Let $A', A''\in \mathcal{A}_i$ and let $B'=F(A')[-i],
B''=F(A'')[-i].$ Without loss of generality, assume that $i$ is
even. Then

\begin{eqnarray} \nonumber \\
F_*(1\otimes [A'])F_*(1\otimes [A''])&=& \langle B', B' \rangle ^i
\langle B'', B'' \rangle ^i (1\otimes [B']k_{B'}^i)(1\otimes
[B'']k_{B''}^i) \nonumber \\
&=& \langle B', B' \rangle ^i \langle B'', B'' \rangle ^i (B'|
B'')^i (1\otimes [B'][B'']k_{B'+B''}) \nonumber \\
&=& \langle B'+B'', B'+B'' \rangle^i (1\otimes [B'][B'']k_{B'+B''}).
\nonumber \\ \nonumber
\end{eqnarray}
Since $F$ is a derived equivalence, it defines a bijection between
the short exact sequences
\[0\rightarrow A'' \rightarrow A \rightarrow A' \rightarrow 0\]
and
\[0\rightarrow B'' \rightarrow B \rightarrow B' \rightarrow 0,\]
and the last equation is equal to the image of
\[\frac{1}{\langle A'', A' \rangle} \sum_{[A]\in \mathcal{I}} g_{A',
A''}^A (1\otimes [A])\] under $F_*.$ This shows that $F_*$ is
consistent with (5.5).

\underline{Case 2: $|i-j|=1$}

Let $A_i\in \mathcal{A}_i$ and $A_{i+1}\in \mathcal{A}_{i+1}.$
Without loss of generality, we may assume that $i$ is even. It is
immediate that $F_*$ preserves (5.5) in the case that $A_i=A$ and
$A_{i+1}=B.$ So, substitute $A_{i+1}=A$ and $A_i=B$ into (5.5). Then
we can rewrite the right hand side of (5.5) in normal form:
\begin{eqnarray} \nonumber
\frac{1}{\langle A_i, A_{i+1} \rangle} \sum_{[C]} g_{A_{i+1}, A_i}^C
(1\otimes [C]) &=& \frac{1}{\langle A_i, A_{i+1} \rangle}
\sum_{[D]\in \mathcal{I}_i, [E]\in \mathcal{I}_{i+1}} g_{A_{i+1},
A_i}^{D\oplus E} (1\otimes [D\oplus E])\\ \nonumber &=&
\frac{1}{\langle A_i, A_{i+1} \rangle} \sum_{[D]\in \mathcal{I}_i,
[E]\in \mathcal{I}_{i+1}} \langle E, D
\rangle g_{A_{i+1}, A_i}^{D\oplus E} (1\otimes [D])(1\otimes [E]). \\
\nonumber
\end{eqnarray}
Here $\mathcal{I}_i$ and $\mathcal{I}_{i+1}$ denote the set of
isomorphism classes of objects in $A_i$ and $A_{i+1},$ respectively.
Now,

\begin{eqnarray} \nonumber
F_*(1\otimes [A_{i+1}])F_*(1\otimes [A_i]) &=& \langle B_{i+1},
B_{i+1} \rangle ^{i+1} \langle B_i, B_i \rangle ^i ([B_{i+1}]
k_{B_{i+1}}^{i+1}
\otimes [B_i] k_{B_i}^i) \nonumber \\
&=& \frac{\langle B_{i+1}, B_{i+1} \rangle ^{i+1} \langle B_i, B_i
\rangle ^i}{(B_{i+1}| B_i)^{i+1}}([B_{i+1}]\otimes
[B_i])(k_{B_{i+1}}^{i+1}\otimes k_{B_i}^i), \nonumber
\end{eqnarray}
where $B_{i+1}=F(A_{i+1})[-i-1]$ and $B_i=F(A_i)[-i].$ Using
\eqref{eq3}-\eqref{eq4}, we have

\begin{eqnarray} \nonumber
[B_{i+1}]\otimes [B_i] &=& \sum_{[L],[N]} \frac{\langle B_i,
B_{i+1}\rangle \langle B_{i+1}, B_{i+1} \rangle}{\langle N, L
\rangle \langle L, L \rangle} \frac{g_{B_{i+1}[1], B_i}^{L[1]\oplus
N}}{|\mbox{Ext}^1(N,L)|} (1\otimes k_{B_{i+1}-L}[N])([L]\otimes 1)
\\ \nonumber
&=& \sum_{[L],[N]} \frac{\langle B_i, B_{i+1}\rangle \langle
B_{i+1}, B_{i+1} \rangle}{\langle N, L \rangle \langle L, L \rangle}
\frac{g_{B_{i+1}[1], B_i}^{L[1]\oplus
N}}{|\mbox{Ext}^1(N,L)|}\frac{(B_{i+1}|N)(L|L)}{(L|N)(B_{i+1}|L)}
(1\otimes [N])([L]\otimes k_{B_{i+1}-L}). \\ \nonumber
\end{eqnarray}
Since $N\in \mathcal{B}_i$ and $L\in \mathcal{B}_{i+1}$ for all
non-zero terms, $|\mbox{Ext}^1(N,L)|=\langle N, L \rangle^{-2},$ and
after simplifying, we have

\[[B_{i+1}]\otimes [B_i] = \sum_{[L],[N]} \frac{\langle B_i, B_{i+1}
\rangle (B_{i+1}|B_i) \langle L, L \rangle}{\langle L, N \rangle
\langle B_{i+1}, B_{i+1} \rangle} g_{B_{i+1}[1], B_i}^{L[1]\oplus N}
(1\otimes [N])([L]\otimes k_{B_{i+1}-L}).\]

So, we calculate

\begin{eqnarray} \nonumber
F_*(1\otimes [A_{i+1}])F_*(1\otimes [A_i]) &=&
\\ \nonumber \frac{\langle B_{i+1}, B_{i+1} \rangle ^{i+1} \langle B_i, B_i \rangle
^i}{(B_{i+1}| B_i)^{i+1}} \sum_{[L],[N]} \frac{\langle B_i, B_{i+1}
\rangle (B_{i+1}|B_i) \langle L, L \rangle}{\langle L, N \rangle
\langle B_{i+1}, B_{i+1} \rangle} g_{B_{i+1}[1], B_i}^{L[1]\oplus N}
(1\otimes [N])([L]k_L^{i+1}\otimes k_N^i) &=& \\ \nonumber
\frac{\langle B_{i+1}, B_{i+1} \rangle ^i \langle B_i, B_i \rangle
^i}{(B_{i+1}| B_i)^i} \sum_{[L],[N]} \frac{\langle B_i, B_{i+1}
\rangle \langle L, L \rangle}{\langle L, N \rangle} g_{B_{i+1}[1],
B_i}^{L[1]\oplus N} (1\otimes [N])([L]k_L^{i+1}\otimes k_N^i) &=& \\
\nonumber \sum_{[L],[N]} \frac{\langle B_{i+1}, B_{i+1} \rangle ^i
\langle B_i, B_i \rangle ^i}{(B_{i+1}| B_i)^i} \frac{\langle B_i,
B_{i+1} \rangle \langle L, L \rangle}{\langle L, N \rangle}
g_{B_{i+1}[1], B_i}^{L[1]\oplus N} (1\otimes [N])([L]k_L^{i+1}\otimes k_N^i) &=& \\
\nonumber \sum_{[L],[N]} \langle L-N, L-N \rangle^i \frac{\langle
B_i, B_{i+1} \rangle \langle L, L \rangle}{\langle L, N \rangle}
g_{B_{i+1}[1], B_i}^{L[1]\oplus N} (1\otimes
[N])([L]k_L^{i+1}\otimes k_N^i) &=& \\ \nonumber \sum_{[L],[N]}
\frac{\langle L, L \rangle^{i+1} \langle N, N \rangle^i \langle B_i,
B_{i+1} \rangle}{\langle L, N \rangle (L|N)^i}g_{B_{i+1}[1],
B_i}^{L[1]\oplus N}(1\otimes [N])([L]k_L^{i+1}\otimes k_N^i) &=& \\
\nonumber \frac{1}{\langle A_i, A_{i+1} \rangle} \sum_{[L],[N]}
\frac{\langle L, L \rangle^{i+1} \langle N, N \rangle^i}{\langle L,
N \rangle}g_{B_{i+1}[i+1], B_i[i]}^{L[i+1]\oplus N[i]}(1\otimes
[N]k_N^i)([L]k_L^{i+1}\otimes 1) &=& \\ \nonumber \frac{1}{\langle
A_i, A_{i+1} \rangle} \sum_{[D]\in \mathcal{I}_i, [E]\in
\mathcal{I}_{i+1}} \langle E, D \rangle g_{A_{i+1}, A_i}^{D\oplus E}
F_*(1\otimes [D])F_*(1\otimes [E]) &=& \\ \nonumber \frac{1}{\langle
A_i, A_{i+1} \rangle} \sum_{[C]} g_{A_{i+1}, A_i}^C F_*(1\otimes
[C]),
\end{eqnarray}
showing that $F_*$ is consistent with (5.5).

\underline{Case 3: $|i-j|\ge 2$}

If $A_i\in \mathcal{A}_i$ and $A_{i+n}\in \mathcal{A}_{i+n}$ where
$n\ge 2,$ then it is clear from the discussion in Section 5 that the
right hand side of (5.5) is the same whether $A_i=A$ and $A_{i+n}=B$
or $A_{i+n}=A$ and $A_i=B.$ Explicitly,

\[\frac{1}{\langle A_{i+n}, A_i \rangle} \sum_{[C]} g_{A_{i},
A_{i+n}}^C (1\otimes [C])=\frac{1}{\langle A_i, A_{i+n} \rangle}
\sum_{[C]} g_{A_{i+n}, A_i}^C (1\otimes [C])=1\otimes [A_i\oplus
A_{i+n}]\] If $n$ is even, then it is also clear that
\[F_*(1\otimes [A_i])F_*(1\otimes [A_{i+n}])=F_*(1\otimes [A_{i+n}])F_*(1\otimes
[A_i])=F_*(1\otimes [A_i\oplus A_{i+n}]).\] If $n$ is odd, then we
have
\begin{eqnarray} \nonumber
F_*(1\otimes [A_i])F_*(1\otimes [A_{i+n}])&=&\langle B_i, B_i
\rangle^i \langle B_{i+n}, B_{i+n} \rangle^j ([B_i]k_{B_i}^i\otimes
1)(1\otimes [B_{i+n}]k_{B_{i+n}}^{i+n})\\ \nonumber &=& \langle B_i,
B_i \rangle^i \langle B_{i+n}, B_{i+n} \rangle^{i+n} ([B_i]\otimes
1)(1\otimes [B_{i+n}])(k_{B_i}^i\otimes k_{B_{i+n}}^{i+n}) \\
\nonumber
\end{eqnarray}
and
\begin{eqnarray} \nonumber
F_*(1\otimes [A_{i+n}])F_*(1\otimes [A_i])&=&\langle B_{i+n},
B_{i+n} \rangle^{i+n} \langle B_i, B_i \rangle^i (1\otimes
[B_{i+n}]k_{B_{i+n}}^{i+n})([B_i]k_{B_i}^i\otimes 1)\\ \nonumber &=&
\langle B_i, B_i \rangle^i \langle B_{i+n}, B_{i+n} \rangle^{i+n}
(1\otimes
[B_{i+n}])([B_i]\otimes 1)(k_{B_i}^i\otimes k_{B_{i+n}}^{i+n}), \\
\nonumber
\end{eqnarray}
where $B_i=F(A_i)[-i]$ and $B_{i+1}=F(A_{i+1})[-i-1].$

Now, because Hom$(B_{i}, B_{i+n})$=Hom$(B_{i+n}, B_i)$=0, from
\eqref{four}, we see that \eqref{eq1}-\eqref{eq2} reduces to
\[[B_{i+n}]\otimes [B_i]=(1\otimes [B_i])([B_{i+n}]\otimes 1)\]
when we take $A=B_{i+n}$ and $B=B_i.$

Hence, the equation
\[F_*(1\otimes [A_i])F_*(1\otimes [A_{i+n}])=F_*(1\otimes [A_{i+n}])F_*(1\otimes
[A_i])=F_*(1\otimes [A_i\oplus A_{i+n}])\] holds in general, and
$F_*$ preserves (5.5).

We now prove that $F_*$ preserves (5.10)-(5.11) by again considering
three separate cases.

\underline{Case 1: $i=j$}

Without loss of generality, take $A,B \in \mathcal{A}_0.$ We show
that the images of the left hand side and the right hand side of
\eqref{eq3}-\eqref{eq4} under $F_*$ are equal. If
$g_{B[1],A}^{N[1]\oplus L}\ne 0,$ then we can decompose $L$ and $N$
as
\[[L]=\langle L_{1}, L_0 \rangle [L_0] [L_1]\]
and
\[[N]=\langle N_0, N_{-1} \rangle [N_{-1}][N_0]\]
where $L_1\in \mathcal{A}_1, L_0, N_0 \in \mathcal{A}_0,$ and
$N_{-1}\in \mathcal{A}_{-1}.$ Similarly, if $g_{A[1]\oplus
B}^{L[1]\oplus N}\ne 0,$ then we can decompose $L$ And $N$ as
\[[L]=\langle L_0, L_{-1} \rangle [L_{-1}] [L_0]\]
and
\[[N]=\langle N_1, N_{0} \rangle [N_0][N_1],\]
where $L_{-1}\in \mathcal{A}_{-1}, L_0, N_0 \in \mathcal{A}_0,$ and
$N_1\in \mathcal{A}_1.$

Now, set $B_A=F(A)$ and $B_B=F(B).$  On the left hand side, set
$B_{L_0}=F(L_0), B_{L_1}=F(L_1)[-1], B_{N_0}=F(N_0),
B_{N_{-1}}=F(N_{-1})[1],$ and on the right hand side, set $B_A=F(A),
B_{L_{-1}}=F(L_{-1})[1], B_{L_0}=F(L_0), B_{N_0}=F(N_0),
B_{N_1}=F(N_1)[-1].$ Then it remains to show that

\begin{equation} \nonumber
\sum_{[L],[N]} \frac{\langle A, B \rangle \langle B, B
\rangle}{\langle L, N \rangle \langle N, N \rangle}
\frac{g_{B[1],A}^{N[1]\oplus L}}{|\mbox{Ext}^1(L,N)|} \frac{\langle
N, N_{-1} \rangle \langle N_0, L_1 \rangle}{\langle L, L_1 \rangle
\langle L_0, N_{-1} \rangle}
k_{B_A-B_{L_0}-B_{N_{-1}}}([B_{L_0}\oplus B_{N_{-1}}]\otimes
[B_{N_0}\oplus B_{L_1}])=
\end{equation}

\begin{equation} \nonumber
\sum_{[L],[N]} \frac{\langle B, A \rangle \langle A, A
\rangle}{\langle N, L \rangle \langle L, L \rangle} \frac{g_{A[1],
B}^{L[1]\oplus N}}{|\mbox{Ext}^1(N,L)|}\frac{\langle L, L_{-1}
\rangle \langle L_0, N_1 \rangle}{\langle N, N_1 \rangle \langle
N_0, L_{-1} \rangle} (1\otimes k_{B_A-B_{L_0}-B_{N_1}})(1\otimes
[B_{N_0}\oplus B_{L_{-1}}])([B_{L_0}\oplus B_{N_1}]\otimes 1).
\end{equation}
But this is just the identity \eqref{eq3}-\eqref{eq4} with $F(A)$
and $F(B)$ in place of $A$ and $B:$ Because $F$ establishes a
bijection between exact triangles

\[A \rightarrow N[1]\oplus L \rightarrow B[1]
\rightarrow A[1]\]

\[\mbox{(resp. } B \rightarrow L[1]\oplus N \rightarrow A[1]
\rightarrow B[1] \mbox{)}\] and

\[B_A \rightarrow (B_{N_0}\oplus B_{L_1})[1]\oplus (B_{L_0}\oplus B_{N_{-1}}) \rightarrow B_B[1]
\rightarrow B_A[1]\]

\[\mbox{(resp. } B_B \rightarrow (B_{L_0}\oplus B_{N_1})[1]\oplus (B_{N_0}\oplus B_{L_{-1}}) \rightarrow B_A[1]
\rightarrow B_B[1] \mbox{),}\] we have
\[g_{B[1],
A}^{N[1]\oplus L}=g_{B_B[1], B_A}^{(B_{N_0}\oplus B_{L_1})[1]\oplus
(B_{L_0}\oplus B_{N_{-1}})}\]\ and
\[g_{A[1], B}^{L[1]\oplus
N}=g_{B_A[1], B_B}^{(B_{L_0}\oplus B_{N_1})[1]\oplus (B_{N_0}\oplus
B_{L_{-1}})}.\] Moreover, this says that it suffices to show that

\[\frac{\langle A, B \rangle \langle B, B \rangle}{\langle
L, N \rangle \langle N, N \rangle} \frac{g_{B[1],A}^{N[1]\oplus
L}}{|\mbox{Ext}^1(L,N)|} \frac{\langle N, N_{-1} \rangle \langle
N_0, L_1 \rangle}{\langle L, L_1 \rangle \langle L_0, N_{-1}
\rangle}=\]

\[\frac{\langle B_A, B_B \rangle \langle B_B, B_B \rangle g_{B_B[1], B_A}^{(B_{N_0}\oplus B_{L_1})[1]\oplus
(B_{L_0}\oplus B_{N_{-1}})}}{\langle B_{L_0}\oplus B_{N_{-1}},
B_{N_0}\oplus B_{L_1} \rangle \langle B_{N_0}\oplus B_{L_1},
B_{N_0}\oplus B_{L_1}\rangle |\mbox{Ext}^1(B_{L_0}\oplus B_{N_{-1}},
B_{N_0}\oplus B_{L_1})|}\] and

\[\frac{\langle B, A \rangle \langle A, A \rangle}{\langle
N, L \rangle \langle L, L \rangle} \frac{g_{A[1],N}^{L[1]\oplus
N}}{|\mbox{Ext}^1(N,L)|} \frac{\langle L, L_{-1} \rangle \langle
L_0, N_1 \rangle}{\langle N, N_1 \rangle \langle N_0, L_{-1}
\rangle}=\]

\[\frac{\langle B_B, B_A \rangle \langle B_A, B_A \rangle g_{B_A[1], B_B}^{(B_{L_0}\oplus B_{N_1})[1]\oplus
(B_{N_0}\oplus B_{L_{-1}})}}{\langle B_{N_0}\oplus B_{L_{-1}},
B_{L_0}\oplus B_{N_1} \rangle \langle B_{L_0}\oplus B_{N_1},
B_{L_0}\oplus B_{N_1} \rangle |\mbox{Ext}^1(B_{N_0}\oplus
B_{L_{-1}}, B_{L_0}\oplus B_{N_1})|}.\] for arbitrary $A,B,L,N\in
\mathcal{A},$ which can be seen by a straightforward computation.

\underline{Case 2: $|i-j|=1$}

Take $A_i\in \mathcal{A}_i, A_{i+1}\in \mathcal{A}_{i+1}$. Without
loss of generality we may assume that $i$ is even. Then for
$A=A_{i+1}$ and $B=A_i,$ (5.10)-(5.11) becomes

\begin{eqnarray} \nonumber
(1\otimes [A_i])([A_{i+1}]\otimes 1)&=& \sum_{[L], [N]}
\frac{\langle A_{i+1}, A_i \rangle \langle A_i, A_i \rangle}{\langle
L, N \rangle \langle N, N \rangle} \frac{g_{A_i[1],
A_{i+1}}^{N[1]\oplus L}}{|\mbox{Ext}^1(L,N)|}
k_{A_{i+1}-L}[L]\otimes [N] \\ \nonumber &=& \sum_{[L], [N]}
\frac{\langle A_{i+1}, A_i \rangle \langle A_i, A_i \rangle}{\langle
L, N \rangle \langle N, N \rangle} g_{A_i[1], A_{i+1}}^{N[1]\oplus
L} \langle L, N \rangle^2 k_{A_{i+1}-L}[L]\otimes [N] \\ \nonumber
&=& \sum_{[L], [N]} \frac{\langle A_{i+1}, A_i \rangle \langle A_i,
A_i \rangle}{\langle N, N \rangle} g_{A_i[1], A_{i+1}}^{N[1]\oplus
L} \langle L, N \rangle \frac{(A_{i+1}|L)(L|N)}{(L|L)(A_{i+1}|N)}
[L]\otimes [N]k_{A_{i+1}-L}^{-1} \\ \nonumber &=& \sum_{[L], [N]}
\frac{\langle A_{i+1}, A_i \rangle \langle A_i, A_i \rangle}{\langle
N, N \rangle} g_{A_i[1], A_{i+1}}^{N[1]\oplus L} \langle L, N
\rangle \frac{(A_{i+1}|A_{i+1})(L|A_i)}{(L|A_{i+1})(A_{i+1}|A_i)}
[L]\otimes [N]k_{A_{i+1}-L}^{-1}. \\ \nonumber
\end{eqnarray}

Let $B_L=F(L)[-i-1], B_N=F(N)[-i], B_i=F(A_i)[-i],$ and
$B_{i+1}=F(A_{i+1})[-i-1].$ Then applying $F_*$ to the right hand
side gives

\begin{eqnarray} \nonumber
&&\sum_{[L], [N]} \frac{\langle A_{i+1}, A_i \rangle \langle A_i,
A_i \rangle}{\langle N, N \rangle} g_{A_i[1], A_{i+1}}^{N[1]\oplus
L} \langle L, N \rangle
\frac{(A_{i+1}|A_{i+1})(L|A_i)}{(L|A_{i+1})(A_{i+1}|A_i)}
\frac{\langle L,L \rangle^{i+1} \langle N,N \rangle^i
}{(L|N)^{i+1}}(1\otimes
[B_L][B_N]k_{B_i}^ik_{B_{i+1}}^{i+1})\nonumber \\
&=& \sum_{[L],[N]} \frac{\langle A_i, A_i \rangle^i \langle A_{i+1},
A_{i+1} \rangle^{i+1} \langle A_{i+1}, A_i \rangle}{(A_i| A_{i+1})^i
\langle N, L \rangle}g_{A_i[1], A_{i+1}}^{N[1]\oplus L} (1\otimes
[B_L][B_N]k_{B_i}^ik_{B_{i+1}}^{i+1}) \nonumber \\
&=& \sum_{[L],[N]} \frac{\langle A_i, A_i \rangle^i \langle A_{i+1},
A_{i+1} \rangle^{i+1} \langle A_{i+1}, A_i \rangle}{(A_i|
A_{i+1})^i}g_{A_i[1], A_{i+1}}^{N[1]\oplus L} (1\otimes [B_L \oplus
B_N]k_{B_i}^ik_{B_{i+1}}^{i+1}) \nonumber \\
&=& \sum_{[B]} \frac{\langle B_i, B_i \rangle^i \langle B_{i+1},
B_{i+1} \rangle^{i+1} (B_i| B_{i+1})^i}{\langle B_{i+1}, B_i
\rangle} g_{B_i, B_{i+1}}^{B} (1\otimes [B]k_{B_i}^ik_{B_{i+1}}^{i+1}) \nonumber \\
&=& \langle B_i, B_i \rangle^i \langle B_{i+1}, B_{i+1}
\rangle^{i+1}(1\otimes
[B_i]k_{B_i}^i[B_{i+1}]k_{B_{i+1}}^{i+1}) \nonumber \\
&=& F_*(1\otimes [A_i])F_*([A_{i+1}]\otimes 1) \nonumber \\
&=& F_*[(1\otimes [A_i])([A_{i+1}]\otimes 1)], \nonumber
\end{eqnarray}
proving that (5.10) and (5.11) have the same image under $F_*.$ The
proof for $A=A_i$ and $B=A_{i+1}$ is similar.

\underline{Case 3: $|i-j|\ge 2$}

If we take $A\in \mathcal{A}_i$ and $B\in \mathcal{A}_j$ for
$|i-j|\ge 2,$ then (5.10)-(5.11) becomes
\[[A]\otimes [B]=(1\otimes
[B])([A]\otimes 1),\] and it is easy to see that
\[F_*([A]\otimes [B])=F_*(1\otimes
[B])F_*([A]\otimes 1).\] Hence, (5.10)-(5.11) is preserved.

\end{proof}

It is clear from the definition of $F_*$ that if
$G:D^b(\mathcal{B})\rightarrow D^b(\mathcal{A})$ is a quasi-inverse
to the derived equivalence $F:D^b(\mathcal{A})\rightarrow
D^b(\mathcal{B}),$ then $G_*$ defines an inverse to $F_*,$ proving
Theorem 1.

\section{Bernstein-Gelfand-Ponomarev reflection functors}

For a quiver $\vec{Q},$ a vertex $\alpha$ is called a source (resp.,
sink) if no arrows end in $\alpha$ (resp., start at $\alpha$). For a
given vertex $\alpha,$ we denote by $\sigma_{\alpha}\vec{Q}$ the
quiver obtained from $\vec{Q}$ by reversing all arrows at $\alpha.$
Now, assume that $\vec{Q}$ contains no multiple edges or oriented
cycles. In this case $\vec{Q}$ must contain at least one source and
sink. The derived Bernstein-Gelfand-Ponomarev reflection functor
(\cite{GM})
\[R_{\alpha}^+:D^b(Rep_k(\vec{Q}))\rightarrow
D^b(Rep_k(\sigma_{\alpha}\vec{Q}))\] is defined for a given source
$\alpha$ as follows: given a complex $C_{\bullet}\in
D^b(Rep_k{\vec{Q}}),$ we define $A_{\bullet}$ to be the subcomplex
of $C_{\bullet}$ corresponding to the $\alpha$ component, i.e.
$(A_i)_{\alpha}=(C_i)_{\alpha}$ and $(A_i)_{\beta}=0$ if $\beta\ne
\alpha.$ Define $B_{\bullet}$ by $(B_i)_{\beta}=(C_i)_{\beta}$ for
$\beta\ne \alpha$ and $(B_i)_{\alpha}=\oplus_{\beta\in S}
(C_i)_{\beta}$ where $S$ is the set of all vertices in $\vec{Q}$
connected to $\alpha$ by an arrow. In addition, we replace each
arrow starting at $\alpha$ with the inclusion map going in the
opposite direction. With the differentials defined in the natural
way, $A_{\bullet}$ and $B_{\bullet}$ become objects in
$D^b(Rep_k(\sigma_{\alpha}\vec{Q})).$ There is a natural map
$f:A_{\bullet}\rightarrow B_{\bullet}$ obtained by taking the direct
sum of the maps $f_{\beta}:(C_i)_{\alpha}\rightarrow (C_i)_{\beta}$
over $S$ in each degree. We define $R_{\alpha}^+(C_{\bullet})$ to be
the mapping cone $C(f)=A_{\bullet}[1]\oplus B_{\bullet}$ of
$f:A_{\bullet}\rightarrow B_{\bullet}.$ It is not hard to check that
this definition is functorial. There is a dual construction
$R_{\alpha}^-$ for reversing the arrows at a sink, and in fact this
defines a quasi-inverse to the derived equivalence $R_{\alpha}^+.$

The category of representations of a quiver is one of the most basic examples of a hereditary Abelian category. Over a finite field, it is easy to check that it is finitary. Applying Theorem 1 to the derived equivalences $R_{\alpha}^{\pm}$
for $k=\mathbb{F}_q$ provides isomorphisms between the double Hall
algebras of quivers with the same underlying graph but different
orientations (as shown in the original paper \cite{BGP}, any
orientation on a graph without cycles can be obtained from
any other orientation by a sequence of reflections at sources and
sinks). These formulas were given in \cite{SVdB} and \cite{XY} using
the original (non-derived) Bernstein-Gelfand-Ponomarev reflection
functors (\cite{BGP}) with $\vec{Q}$ an A-D-E quiver. In this case,
the double Hall algebra of $Rep_{\mathbb{F}_q}(\vec{Q})$ is
canonically isomorphic to $U_q(\mathfrak{g})$ and these formulas
give automorphisms of $U_q(\mathfrak{g})$ which turn out to coincide
with Lusztig's symmetries (\cite{L}, \cite{XY}). Moreover, it is known that the
Grothendieck groups $K_0(D^b(Rep_{\mathbb{F}_q}(\vec{Q}))$ and
$K_0(D^b(Rep_{\mathbb{F}_q}(\sigma_{\alpha}\vec{Q}))$ can be
identified with the root lattice of $U_q(\mathfrak{g})$, and
under this correspondence the induced map
\[(R_{\alpha}^+)_*:K_0(D^b(Rep_{\mathbb{F}_q}(\vec{Q})))\rightarrow
K_0(D^b(Rep_{\mathbb{F}_q}(\sigma_{\alpha}\vec{Q})))\] is exactly
the reflection at the simple root associated to $\alpha.$

\section{The Kronecker quiver and $Coh(\mathbb{P}^1)$}
Let $K$ denote the Kronecker quiver, i.e.the quiver with two
vertices and two arrows oriented in the same direction. The
underlying graph of $K$ is the affine Dynkin diagram $A_1^{(1)},$ so
$K$ is tame and the dimension vectors of its indecomposable
representations correspond to the positive roots of
$\hat{\mathfrak{sl}_2}.$ In dimension $(n,n+1)$ ($n\in
\mathbb{Z}_+$) there is a unique indecomposable object up to
isomorphism, which we denote by $P(n).$ It is given by two maps $
\left(
\begin{array}{ccc}
I_n   \\
0   \end{array} \right)$ and $ \left( \begin{array}{ccc}
0   \\
I_n   \end{array} \right)$ from $k^n$ to $k^{n+1}.$ Similarly, in
dimension $(n+1,n)$ ($n\in \mathbb{Z}_+$) there is a unique
indecomposable object up to isomorphism given by the maps $ \left(
\begin{array}{ccc} I_n & 0
\end{array} \right)$ and $ \left( \begin{array}{ccc} 0 & I_n
\end{array} \right)$ from $k^{n+1}$ to $k^n$ which we denote by $I(n).$ These two classes of indecomposable
objects form the preprojective and preinjective components of
$Rep_k(K),$ respectively. The regular indecomposable objects lie in
dimensions $(n,n)$ corresponding to the imaginary roots and for each $n$ they are
parameterized by $x\in \mathbb{P}^1(k).$

Now consider $Coh(\mathbb{P}^1(k)),$ the category of coherent
sheaves on $\mathbb{P}^1(k).$ For this category the indecomposable
objects can be parameterized by the nonstandard set of positive
roots
\[\{(n,n+1)|n\in \mathbb{Z}\} \cup \{(n,n)|n\in \mathbb{N}\}\]
of $\hat{\mathfrak{sl_2}}.$ The correspondence is given by defining
\[\underline{dim}(\mathcal{F})=(deg(\mathcal{F}), rank(\mathcal{F})+deg(\mathcal{F}))\] The line bundles
$\mathcal{O}(n)$ then correspond to the real roots $(n,n+1)$ and
corresponding to each imaginary root $(n,n)$ is the family of
indecomposable sheaves $\mathcal{O}_x^n$ parameterized by $x\in
\mathbb{P}^1(k).$

The conditions of Definition 1 apply to $Coh(\mathbb{P}^1(k))$ when $k$ is a finite field due to the following lemma.

\begin{lemma}
Let $X$ be a smooth projective curve over a finite field $k.$ Then $Coh(X)$ is finitary.
\end{lemma}

\begin{proof}
It is a theorem of Serre that $\mbox{Hom}(\mathcal{F},\mathcal{G})$ and $\mbox{Ext}(\mathcal{F},\mathcal{G})$ are finite-dimensional vector spaces over $k$ for all $\mathcal{F},\mathcal{G}\in Coh(X)$ (\cite{Har}). The map $d:K_0(Coh(X))\rightarrow \mathbb{Z}^2$ defined by $d([\mathcal{F}])=(rank(\mathcal{F}), deg(\mathcal{F}))$ satisfies the desired properties of Definition 1.

\end{proof}

In addition to the indecomposable objects of $Coh(\mathbb{P}^1(k))$
and $Rep_k(K)$ both being parameterized by positive roots of
$\hat{\mathfrak{sl_2}},$ the Hall algebras of these two categories (which we denote by $H_{\mathbb{P}^1}$ and
$H_K,$ respectively) are related to the quantum enveloping algebras
of the Borel subalgebras of $\hat{\mathfrak{sl_2}}$ associated to
their corresponding set of positive roots. Specifically, the
subalgebra of $H_K$ generated by the elements $[P(0)],$ $[I(0)],$
$k_{P(0)}^{\pm 1},$ and $k_{I(0)}^{\pm 1},$ which we denote by
$C_K,$ is isomorphic to $U_q(\mathfrak{b}_+),$ the Borel subalgebra
corresponding to the affine root system of $\hat{\mathfrak{sl_2}}.$
The subalgebra of $H_{\mathbb{P}^1}$ generated by the elements
$[\mathcal{O}(n)],$ $T_r,$ $k_{\mathcal{O}},$ and
$k_{\mathcal{O}_x},$ which we denote by $C_{\mathbb{P}^1},$ is
isomorphic to $U_q(\mathfrak{b}_+^{\mathcal{L}})$ in the Drinfeld
realization \cite{D1} of $U_q(\hat{\mathfrak{sl_2}}).$ Here
$\mathfrak{b}_+^{\mathcal{L}}=\mathfrak{n}^{\mathfrak{sl_2}}_+[t,t^{-1}]\oplus
\delta^{\mathfrak{sl_2}}[t]$ and $T_r$ denotes a specific linear
combination of elements $[\mathcal{O}_x]$ over $x\in
\mathbb{P}^1(k)$ depending on the integer $r\in \mathbb{Z}_+$ (see
\cite{BaK} for details).

A result of \cite{Be} is the derived
equivalence $D^b(Coh(\mathbb{P}^1(k)))\cong D^b(Rep_k (K)).$ This is
given explicitly by the tilting functor

\begin{equation} \la{rhom}
RHom(\mathcal{O}\oplus \mathcal{O}(1),\cdot):
D^b(Coh(\mathbb{P}^1(k)))\rightarrow D^b(Rep_k (K))
\end{equation}
which sends the line bundles $\mathcal{O}(n)$ to the preprojective
indecomposables $P(n)$ for $n\ge 0,$ and to the preinjective
indecomposables $I(-n-1)$ for $n<0$ (translated to the right by one
degree), and sends the torsion sheaves $\mathcal{O}_x^n$ to the
regular indecomposables in dimension $(n,n).$ By Theorem 1, this equivalence induces an isomorphism between the double Hall algebras $DH_{\mathbb{P}^1}$ and $DH_{K}.$ Let $DC_{\mathbb{P}^1}$ denote the subalgebra of $DH_{\mathbb{P}^1}$ generated by the elements $a\otimes 1$ and $1\otimes b$ for $a\in C_{\mathbb{P}^1}$ and $b\in C^{coop}_{\mathbb{P}^1}.$ Let $DC_{K}$ denote the subalgebra of $DH_{K}$ generated by the elements $a\otimes 1$ and $1\otimes b$ for $a\in C_{K}$ and $b\in C^{coop}_{K}.$  As shown in \cite{BS}, the isomorphism between $DH_{\mathbb{P}^1}$ and $DH_{K}$ restricts to an isomorphism between $DC_{\mathbb{P}^1}$ and $DC_{K},$ which can be identified with the Drinfeld-Beck isomorphism $U_q(\mathcal{L}\mathfrak{sl_2})\rightarrow U_q(\hat{\mathfrak{sl_2}}).$ 

\section{Affine quivers and weighted projective lines}
As a generalization of $Coh(\mathbb{P}^1),$ we can consider the category of coherent sheaves on a weighted projective line $\mathbb{X}_{p,\underline{\lambda}},$ first studied in \cite{GL}. Given a sequence of postive integers $\bold{p}=(p_1, \cdots, p_n),$ we define the weighted projective space $\mathbb{P}_{\bold{p}}(k)$ as the quotient of $\mathbb{A}_n-\{0\}$ by the action of the group $G(\bold{p})=\{(t_1,\cdots t_n)\in (k^*)^n | t_1^{p_1}=\cdots = t_n^{p_n}\}$. For a sequence of distinct elements $\underline{\lambda}=(\lambda_1,\cdots, \lambda_{n-1})\in (\mathbb{P}^1(k))^n,$ the equations $X_i^{p_i}=X_2^{p_2}-\lambda_i X_1^{p_1},$ for $i=3,\cdots, n-1$ determine a two-dimensional subvariety $F(\bold{p}, \underline{\lambda})$ of $\mathbb{A}_n.$ We define the weighted projective line $\mathbb{X}_{p, \underline{\lambda}}$ as the quotient of $F(\bold{p}, \underline{\lambda})$ by $G(\bold{p}).$ To define a sheaf theory for $\mathbb{X}_{p, \underline{\lambda}},$ we first let $L(\bold{p})$ denote the rank one abelian group on generators $\vec{x_1}, \cdots, \vec{x_n}$ with relations $p_1\vec{x_1}=\cdots =p_n\vec{x_n}.$ Then the coordinate ring $S(\bold{p},\underline{\lambda})=k[x_1,\cdots, x_n]/I_{\lambda}$ of the affine variety $F(\bold{p}, \underline{\lambda})$ can be given an $L(\bold{p})$-grading by setting $deg(x_i)=\vec{x_i}.$ The category  $Coh(\mathbb{X}_{p, \underline{\lambda}})$ of coherent sheaves on $\mathbb{X}_{p, \underline{\lambda}}$ is defined as the Serre quotient of the category of finitely generated graded $S(\bold{p},\underline{\lambda})$-modules by the Serre subcategory of finite dimensional modules.

The ``virtual genus" of $\mathbb{X}_{p,\underline{\lambda}}$ is defined to be \[g_{\mathbb{X}}=1+\frac{1}{2}\left((n-2)p-\sum_{i=1}^n \frac{p}{p_i}\right)\]
where $p$ is the least common multiple of $p_1,\cdots, p_n.$ It satisfies an analogue of the Riemann-Roch theorem. We associate a graph $\mathbb{T}_{p_1,\cdots, p_n}$ to the sequence $(p_1, \cdots, p_n)$ as follows. Starting with a central vertex, we attach to it $n$ segments, the $i$-th segment consisting of $p_i-1$ edges.

\begin{example}
For $p_1,p_2\ge 2,$ the graph $\mathbb{T}_{p_1,p_2}$ is $A_{p_1+p_2-1}.$

\end{example}

\begin{example}
For $p\ge 2,$ the graph $\mathbb{T}_{p,2,2}$ is $D_{p+2}.$

\end{example}

It can be checked that $\mathbb{T}_{p_1,\cdots, p_n}$ is an A-D-E Dynkin diagram if and only if $g_{\mathbb{X}}<1,$ and if $g_{\mathbb{X}}=1,$ then it is an affine A-D-E diagram. In general, we can associate a Kac-Moody algebra $\mathfrak{g}$ to $\mathbb{T}_{p_1,\cdots, p_n}.$ The Grothendieck group $K_0(Coh(\mathbb{X}_{p,\underline{\lambda}}))$ is isomorphic to the root lattice of the loop algebra $\mathcal{L}\mathfrak{g}.$ As in the case of quivers and $Coh(\mathbb{P}^1),$ the class of an indecomposable object in $K_0(Coh(\mathbb{X}_{p,\underline{\lambda}}))$ always corresponds to a positive root (see \cite{CB}). Note that $Coh(\mathbb{P}^1)$ may be seen as a weighted projective line with $\bold{p}=(1)$ and $\mathbb{T}_1=A_1.$

\begin{proposition}
The category $Coh(\mathbb{X}_{p,\underline{\lambda}})$ is hereditary and abelian, and if $k$ is finite it is finitary. 
\end{proposition}

\begin{proof}
It is shown that $Coh(\mathbb{X}_{p,\underline{\lambda}})$ is hereditary and abelian in \cite{GL}. It is also shown that $Hom(\mathcal{F}, \mathcal{G})$ is finite for all $\mathcal{F},\mathcal{G}\in Coh(\mathbb{X}_{p,\underline{\lambda}}),$ and by Serre duality this holds for $Ext^1(\mathcal{F}, \mathcal{G})$ as well. Finally, to define $d:K_0(Coh(\mathbb{X}_{p,\underline{\lambda}}))\rightarrow \mathbb{Z}^2,$ we again use the map $d([\mathcal{F}])=(rank(\mathcal{F}), deg(\mathcal{F})).$

\end{proof}

We can therefore define the double Hall algebra of $Coh(\mathbb{X}_{p,\underline{\lambda}}),$ which we denote by $DH_{p,\underline{\lambda}}.$ As shown by Schiffmann \cite{Sc2}, in the case of $g_{\mathbb{X}}=0$ or $1,$ the Hall algebra $H_{p,\underline{\lambda}}$ of $Coh(\mathbb{X}_{p,\underline{\lambda}})$ contains a subalgebra isomorphic to $U_q(\mathfrak{b}_+^{\mathcal{L}})$ in the Drinfeld realization of $U_q(\hat{\mathfrak{g}})$. Here $\mathfrak{g}$ is the Kac-Moody algebra corresponding to $\mathbb{T}_{p_1\, \cdots, p_n}.$ The equivalence \eqref{rhom} can be generalized by replacing $\mathcal{O}\oplus \mathcal{O}(1)$ with $\bigoplus_{0\le \vec{x} \le \vec{c}} \mathcal{O}(\vec{x})$ where $\vec{c}=p_1\vec{x_1}=\cdots=p_n\vec{x_n}$ and $\mathcal{O}(\vec{x})$ is the image of the $S(\bold{p},\underline{\lambda})$-module $S(\bold{p},\underline{\lambda})[\vec{x}]$ in $Coh(\mathbb{X}_{p,\underline{\lambda}})$. In this case, $Rep_k(K)$ is replaced with the category of finite-dimensional modules over Ringel's ``canonical algebra" $\Lambda_{p, \underline{\lambda}}$ associated to $\bold{p}$ and $\underline{\lambda}$ (see \cite{R4} for details). The algebras $\Lambda_{p, \underline{\lambda}}$ are not hereditary in general, and this equivalence was used to give a classification of indecomposable modules over tubular algebras (i.e., those corresponding to $g_{\mathbb{X}}=1$). In the case that $g_{\mathbb{X}}<1,$ there is another derived equivalence between $\Lambda_{p, \underline{\lambda}}$-mod and the category of representations of the affine quiver corresponding to $\mathbb{T}_{p_1, \cdots, p_n}.$ Applying Theorem 1 to the composition 
\[D^b(Coh(\mathbb{X}_{p,\underline{\lambda}}))\rightarrow D^b(Rep_{\mathbb{F}_q}(\hat{Q}))\]
defines an isomorphism between the double Hall algebra $DH_{p, \lambda}$ of $Coh(\mathbb{X}_{p,\underline{\lambda}})$ and the double Hall algebra $DH_{\hat{Q}}$ of $Rep_{\mathbb{F}_q}(\hat{Q}).$ It seems plausible that the restriction of this isomorphism to the subalgebras corresponding to $U_q(\mathfrak{b}_+^{\mathcal{L}})$ and $U_q(\mathfrak{b}_+)$ can again be identified with the Drinfeld-Beck isomorphism $U_q(\mathcal{L}\mathfrak{g})\rightarrow U_q(\hat{\mathfrak{g}}).$

\bibliographystyle{amsalpha}

\end{document}